\documentclass[a4paper, 10pt]{amsart}
\usepackage[top=80pt,bottom=65pt,left=70pt,right=70pt]{geometry}
\usepackage{graphicx, eepic}
\usepackage{times}
\normalfont
\usepackage{bbm}
\usepackage{xcolor}

\usepackage{amsthm,amsmath,amsfonts,amssymb}
\usepackage{hyperref}
\hypersetup{colorlinks}
\hypersetup{
    bookmarks=true,         % show bookmarks bar?
    unicode=false,          % non-Latin characters in Acrobat’s bookmarks
    pdftoolbar=true,        % show Acrobat’s toolbar?
    pdfmenubar=true,        % show Acrobat’s menu?
    pdffitwindow=false,     % window fit to page when opened
    pdfstartview={FitH},    % fits the width of the page to the window
    pdftitle={My title},    % title
    pdfauthor={Author},     % author
    pdfsubject={Subject},   % subject of the document
    pdfcreator={Creator},   % creator of the document
    pdfproducer={Producer}, % producer of the document
    pdfkeywords={keyword1} {key2} {key3}, % list of keywords
    pdfnewwindow=true,      % links in new PDF window
    colorlinks=true,       % false: boxed links; true: colored links
    linkcolor=blue,          % color of internal links (change box color with linkbordercolor)
    citecolor=red,        % color of links to bibliography
    filecolor=violet,      % color of file links
    urlcolor=violet       % color of external links
}
\usepackage{multirow, url}
\usepackage{color}
\usepackage[all]{xy}
\newtheorem{thm}{Theorem}[section]

\newtheorem{lem}[thm]{Lemma}
\newtheorem{prop}[thm]{Proposition}

\newtheorem{conj}[thm]{Conjecture}

\theoremstyle{definition}
\newtheorem{defn}[thm]{Definition}
\theoremstyle{remark}
\newtheorem{rem}[thm]{Remark}

\numberwithin{equation}{section} \numberwithin{table}{section}

\newcommand{\zell}{{\Z/{\ell\Z}}}
\newcommand{\muell}{{\mu_{\ell}}}

\newcommand{\GQ}{{\mathrm{Gal}(\overline{\mathbb{Q}}/{\mathbb{Q}})}}

\newcommand{\arsub}{\ar@{}[r]|-*[@]{\subset}}
\newcommand{\arsup}{\ar@{}[r]|-*[@]{\supset}}
\newcommand{\arcap}{\ar@{}[d]|-*[@]{\subset}}
\newcommand{\arcup}{\ar@{}[u]|-*[@]{\subset}}

\renewcommand{\pmod}[1]{{~(\mathrm{mod}~{#1})}}

\newcommand{\F}{{\mathbb{F}}}

\newcommand{\Q}{{\mathbb{Q}}}
\newcommand{\Z}{{\mathbb{Z}}}
\newcommand{\R}{{\mathbb{R}}}
\newcommand{\C}{{\mathbb{C}}}
\newcommand{\T}{{\mathbb{T}}}

\newcommand{\bH}{{\mathbb{H}}}

\newcommand{\bP}{{\mathbb{P}}}

\newcommand{\m}{{\mathfrak{m}}}

\newcommand{\cC}{{\mathcal{C}}}
\newcommand{\cD}{{\mathcal{D}}}

\newcommand{\cO}{{\mathcal{O}}}

\newcommand{\cX}{{\mathcal{X}}}

\newcommand{\Hom}{{\mathrm{Hom}}}
\newcommand{\Gal}{{\mathrm{Gal}}}

\newcommand{\Pic}{{\mathrm{Pic}}}
\newcommand{\End}{{\mathrm{End}}}

\newcommand{\Frob}{{\mathrm{Frob}}}

\newcommand{\Ver}{{\mathrm{Ver}}}
\newcommand{\old}{{\mathrm{old}}}
\newcommand{\new}{{\mathrm{new}}}
\newcommand{\Ta}{{\mathrm{Ta}}}
\newcommand{\tor}{{\mathrm{tor}}}

\newcommand{\upto}{{up to products of powers of 2 and 3}}
\newcommand{\modl}{{\pmod {\ell}}} 
\mathchardef\hyp="2D

\newcommand{\mat}[4]{
 \left(  \begin{smallmatrix} #1 & #2 \\ #3 & #4 \end{smallmatrix} \right)}

\newcommand{\br}[1]{\langle #1 \rangle}

\newcommand{\plim}[1]{\lim_{\substack{\longleftarrow\\{#1}}}\;}

\newcommand{\zmod}[1]{{\Z/{#1}\Z}}

\newcommand{\exclude}[1]{}

\begin{document}                                                                          

\title{Rational torsion points on Jacobians of Shimura curves}
\author{Hwajong Yoo}
\address{Center for Geometry and Physics, Institute for Basic Science (IBS), Pohang, Republic of Korea 37673}
\email{hwajong@gmail.com}
\thanks{This work was supported by IBS-R003-G1.}

%\date{\today}

\begin{abstract}
Let $p$ and $q$ be distinct primes. There is the Shimura curve $\cX^{pq}$ associated to the indefinite quaternion algebra of discriminant $pq$ over $\Q$. Let $J^{pq}$ be the Jacobian variety of $\cX^{pq}$, which is an abelian variety over $\Q$. For an odd prime $\ell$, we provide sufficient conditions for the non-existence of rational points of order $\ell$ on $J^{pq}$. As an application, we find some non-trivial subgroups of the kernel of an isogeny from the new quotient $J_0(pq)^{\new}$ of $J_0(pq)$ to $J^{pq}$.
\end{abstract} 

\maketitle
\setcounter{tocdepth}{1}
\tableofcontents

\section{Introduction}
Let $p$ and $q$ be distinct primes. Consider the modular curves $X_0(p)$ and $X_0(pq)$ over $\Q$; and their Jacobians $J_0(p)$ and $J_0(pq)$ over $\Q$. By Mordell-Weil theorem, the rational points on $J_0(p)$ and $J_0(pq)$ are finitely generated abelian groups, and hence
$$
J_0(p)(\Q) \simeq \Z^a \oplus J_0(p)(\Q)_{\tor}\quad\text{and}\quad J_0(pq)(\Q) \simeq \Z^b \oplus J_0(pq)(\Q)_{\tor},
$$
where $J_0(p)(\Q)_{\tor}$ and $J_0(pq)(\Q)_{\tor}$ are finite abelian groups. 

In early 1970s Ogg \cite{Og75} conjectured that the group $J_0(p)(\Q)_{\tor}$ is generated by the cuspidal divisor $[(0)-(\infty)]$. In his landmark paper \cite{M77}, Mazur proved Ogg's conjecture. To do this, he studied submodules of $J_0(p)$ annihilated by the Eisenstein ideal of the Hecke ring of level $p$. A natural generalization is as follows.

\begin{conj}[Generalized Ogg's conjecture]
All rational torsion points on $J_0(pq)$ are cuspidal, i.e.,
$$
J_0(pq)(\Q)_{\tor} = \cC(pq),
$$
where $\cC(pq)$ is the cuspidal group of $J_0(pq)$.
\end{conj}
Our present knowledge is insufficient to prove the above conjecture completely. As Mazur pointed out \cite[p. 34]{M77}, control of the $2$-torsion part of $J_0(pq)(\Q)_{\tor}$ is very difficult. Note that some of this conjecture is now proved by Ohta \cite{Oh14} and by the author \cite{Yoo5}.

In this paper, instead of studying the above conjecture, we consider an abelian variety $J^{pq}$, which is isogenous to the new quotient $J_0(pq)^{\new}$ of $J_0(pq)$. More specifically, let  $\cX^{pq}$ be the Shimura curve associated to the indefinite quaternion algebra over $\Q$ of discriminant $pq$ with trivial level structure. Let $J^{pq}$ be the Jacobian variety of $\cX^{pq}$, which is an abelian variety over $\Q$ of dimension $g(\cX^{pq})$, the genus of $\cX^{pq}$. From now on, we always assume that $g(\cX^{pq}) \neq 0$. Then, we prove the following theorem.

\begin{thm}\label{thm:maintheorem}
For a prime $\ell\geq 5$, the Jacobian $J^{pq}$ does not have rational points of order $\ell$ unless one of the following holds:
\begin{itemize}
\item $p\equiv q \equiv 1 \modl$; 
\item $p\equiv 1 \modl$ and $q^{\frac{p-1}{\ell}} \equiv 1 \pmod p$;
\item $q\equiv 1 \modl$ and $p^{\frac{q-1}{\ell}} \equiv 1 \pmod q$.
\end{itemize}
Furthermore, the Jacobian $J^{pq}$ does not have rational points of order $3$ if $(p-1)(q-1)$ is not divisible by $3$.
\end{thm}

As an application of this theorem, we get information about the kernel of an isogeny between $J_0(pq)^{\new}$ and $J^{pq}$. (The existence of this isogeny is due to Ribet \cite[Th\`eor\`eme 2]{R80}.) More precisely,
let $\Psi(pq)$ denote such an isogeny over $\Q$, and let $K(pq)$ denote the kernel of $\Psi(pq)$.
By the careful study of bad reduction of Shimura curves, Ogg \cite{Og85} conjectured that the image of some cuspidal divisors in $J_0(pq)$ belongs to $K(pq)$. In the case of low genus Shimura curves, this conjecture was proved by Gonz\'alez and Molina \cite{GM11}.
More precisely, they found an equation of $\cX^{pq}$ in the case where $g(\cX^{pq}) \leq 3$ and computed $K(pq)$ (for chosen $\Psi(pq)$) by the consideration of bad reduction of $\cX^{pq}$.
Instead of finding an explicit equation of $\cX^{pq}$ and computing $K(pq)$, we prove that $K(pq)$ always contains $\pi(\cC_{\ell}(pq))$ if $\ell$ satisfies certain conditions, where $\cC_{\ell}(pq)$ is the $\ell$-primary subgroup of $\cC(pq)$ and $\pi$ is the quotient map from $J_0(pq)$ to $J_0(pq)^{\new}$.

\begin{thm}\label{thm:kernelofisogeny}
Let $\ell^m$ and $\ell^n$ be the exact powers of $\ell$ dividing $p+1$ and $q+1$, respectively.
If $\ell\geq 5$ and all the following conditions hold, then $K(pq)$ contains $\pi(\cC_{\ell}(pq))$, and the latter is isomorphic to $\zmod {\ell^m} \oplus \zmod {\ell^n}$:
\begin{itemize}
\item $\ell$ does not divide $(p-1, q-1)$; 
\item if $p\equiv 1 \modl$, then $q^{\frac{p-1}{\ell}} \not\equiv 1 \pmod p$;
\item if $q\equiv 1 \modl$, then $p^{\frac{q-1}{\ell}} \not\equiv 1 \pmod q$.
\end{itemize}

If $\ell=3$ and $(p-1)(q-1)$ is not divisible by $3$, then $K(pq)$ contains $\pi(\cC_3(pq))$, and the latter is isomorphic to $\zmod {3^{\alpha}} \oplus \zmod {3^{\beta}}$,  where $\alpha=\max \{0,~m-1\}$ and $\beta=\max \{0, ~n-1\}$.
\end{thm}

The organization of this article is as follows. In \textsection \ref{sec:Eisenstein}, we discuss all possible new Eisenstein maximal ideals of level $pq$. In \textsection \ref{sec:criteriaformaximality}, we give certain criteria on primes $p$ and $q$ for an Eisenstein ideal discussed in the previous section to be maximal. In \textsection \ref{sec:structure}, we study the structures of the kernels of Eisenstein maximal ideals on Jacobians. In \textsection \ref{sec:proof}, we deduce Theorem \ref{thm:maintheorem} from the above results. Finally, we prove Theorem \ref{thm:kernelofisogeny} in \textsection \ref{sec:kernelofisogeny}. 

\subsubsection*{Acknowledgements} We are grateful to Ken Ribet and Sug Woo Shin for valuable comments and discussions. 

\subsection{Notation}
Let $B$ be a quaternion algebra over $\Q$ of discriminant $D$ such that $\phi : B\otimes_{\Q} \R \simeq M_2(\R)$. Let $\cO$ be an Eichler order of level $N$ of $B$ and let $\cO^{\times, 1}$ be the set of (reduced) norm 1 elements in $\cO$. We define $\Gamma_0^D(N):=\phi(\cO^{\times, 1})$. Let $X_0^D(N)$ be the Shimura curve over $\Q$ associated to $B$ with $\Gamma_0^D(N)$ level structure and let $J_0^D(N):=\Pic^0(X_0^D(N))$ be its Jacobian variety. If $D=1$, then $X_0(N)=X_0^1(N)$ denotes the modular curve for $\Gamma_0(N)$ and $J_0(N)=J_0^1(N)$ denotes its Jacobian variety. If $D\neq 1$, then $X_0^D(N)(\C) \simeq \Gamma_0^D(N)\backslash \bH$, where $\bH$ is the complex upper half plane. 

For an integer $n\geq 1$, there is a Hecke operator $T_n$ acting on $J_0^D(N)$. We denote by $\T^D(N)$ the $\Z$-subalgebra of the endomorphism ring of $J_0^D(N)$ generated by
all $T_n$. In the case where $D=1$ (resp. $N=1$), we simply denote by $\T(N)$ (resp. $\T^D$) the Hecke ring $\T^1(N)$ (resp. $\T^D(1)$). 
If $p$ divides $DN$, we often denote by $U_p$ the $p^{\text{th}}$ Hecke operator $T_p$ on $J_0^D(N)$. 
For a prime $p$ dividing $N$, there is also an Atkin-Lehner involution $w_p$ on $J_0^D(N)$. 
For a maximal ideal $\m$ of a Hecke ring $\T$, we denote by $\T_{\m}$ the completion of $\T$ at $\m$, i.e., 
$$
\T_{\m}:= \plim n \T/{\m^n}.
$$

\section{Eisenstein ideals in $\T^{pq}$}\label{sec:Eisenstein}
From now on, we fix distinct primes $p$ and $q$; and $\ell$ denotes an odd prime. Let $\T:=\T^{pq}$ and $I_0:=(T_r-r-1~:~\text{for primes } r \nmid pq) \subset \T$.

\begin{lem}\label{lem:Up^2=1}
We have $U_p^2=U_q^2=1 \in \T$.
\end{lem}
\begin{proof}
Let $w_p$ and $w_q$ be Atkin-Lehner involutions on $J_0(pq)$. Then $U_p+w_p=0$ on the space of newforms (cf. \cite[Proposition 3.7]{R90}). Because $\Psi(pq)$ is Hecke-equivariant (cf. \cite[\textsection 4]{R90}), $U_p$ and $U_q$ are also involutions on $J^{pq}$.
\end{proof}

\begin{defn}
We define Eisenstein ideals containing $I_0$ as follows: 
$$
I_1 := (U_p-1,~U_q-1~, I_0), \quad\quad\quad\quad\quad I_2 := (U_p+1,~U_q+1,~I_0),
$$
\vspace{-8mm}
$$
I_3 := (U_p-1,~U_q+1~, I_0) \quad\quad\text{and}\quad\quad I_4 := (U_p+1,~U_q-1,~I_0).
$$
Moreover, we set $\m_i :=(\ell, ~I_i)$. They are all possible Eisenstein maximal ideals in $\T$ by the above lemma.
\end{defn}

Let $\T_{\ell}:=\T \otimes_{\Z} \Z_{\ell}$. Then it is a semi-local ring and we have
$$
\T_{\ell} = \prod_{\ell \in \m~ \text{maximal}} \T_{\m}.
$$

Using the above description of Eisenstein maximal ideals, we prove the following.

\begin{prop}\label{prop:decomposition}
The quotient $\T_{\ell}/{I_0}$ decomposes as follows:
$$
\T_{\ell}/{I_0} = \prod_{i=1}^4 \T_{\m_i}/I_0 \simeq \prod_{i=1}^4 \T_{\ell}/{I_i}.
$$
\end{prop}

\begin{proof}
It suffices to prove that $\T_{\ell}/{I} \simeq \T_{\m}/I_0$, where $I:=I_i$ and $\m:=(\ell, ~I)$. 
We discuss the case where $i=1$ and other cases are basically the same. 

If $\m$ is not maximal, then $\T_{\m}=0=\T_{\ell}/I$. Therefore we may assume that $\m$ is maximal. 
Since $\ell$ is odd and $U_p-1 \in \m$, we have $U_p+1 \not\in \m$. In other words, $U_p+1$ is a unit in $\T_{\m}$. By Lemma \ref{lem:Up^2=1}, we have $U_p-1=0$ in $\T_{\m}$ and hence $U_p-1 \in I_0$. Similarly we have $U_q-1 \in I_0$. Therefore $\T_{\m}/{I_0} = \T_{\m}/I$. Since the index of $I$ in $\T$ is finite (cf. \cite[Lemma 3.1]{Yoo14}), there is an integer $n$ such that $\m^n \subseteq I$. Thus, we have $(\T_{\ell}/{\m^n})/I = \T_{\ell}/I$ and hence $\T_{\m}/I \simeq \T_{\ell}/I$.
\end{proof}

\section{Criteria for $\m$ to be maximal}\label{sec:criteriaformaximality}
In this section, we discuss certain conditions on the primes $p$ and $q$ for which $\m_i$ is maximal. By the Jacquet-Langlands correspondence, it suffices to show that $\m_i$ is new maximal in $\T(pq)$ under the given assumption. 

\subsection{Maximality of $\m_1$} \label{sec:s=2}
\begin{thm}\label{thm:m1}
The ideal $\m_1$ is maximal in $\T$ if and only if one of the following holds:
\begin{itemize}
\item $p \equiv q \equiv 1 \modl$;
\item $\ell$ divides the numerator of $\frac{p-1}{3}$ and $q^{\frac{p-1}{\ell}}\equiv 1 \pmod p$;
\item $\ell$ divides the numerator of $\frac{q-1}{3}$ and $p^{\frac{q-1}{\ell}}\equiv 1 \pmod q$.
\end{itemize} 
\end{thm}
\begin{proof}
Let $\m:=\m_1$ and $I:=I_1$.

If $\ell\geq 5$ and $\ell$ does not divide $pq$, this is proved by Ribet \cite[Theorem 2.4]{Yoo14a} (The proof of this theorem is given in \textsection 4 of \textit{op. cit.})

Since the index of $I$ in $\T(pq)$ is equal to the numerator of $\frac{(p-1)(q-1)}{3}$ up to powers of 2 \cite[Theorem 5.1]{Yoo15a}, we assume that $\ell$ divides the numerator of $\frac{(p-1)(q-1)}{3}$, and hence $\m$ is maximal in $\T(pq)$.

Let $\ell=3$ and let $\lambda$ be the ideal of $\T(p)$ corresponding to $\m$. By Mazur \cite{M77}, $\lambda$ is maximal if and only if $p \equiv 1 \pmod 9$. 
First, we assume that $p \equiv 1 \pmod 9$. Let $R:=\T(p)_{\lambda}$ be the completion of $\T(p)$ at $\lambda$, and let $I$ be the Eisenstein ideal of $\T(p)$. 
Then, $IR\neq (T_q-q-1)R$ if and only if either $q\equiv 1 \pmod 3$ or $q^{\frac{p-1}{3}}\equiv 1 \pmod p$. By the same argument as in the proof of \cite[Theorem 2.4]{Yoo14a}, $\m$ is new maximal if and only if $IR \neq (T_q-q-1)R$. Therefore by symmetry, the result follows unless $p-1$ and $q-1$ are exactly divisible by $3$. Next, we assume that $p-1$ and $q-1$ are exactly divisible by $3$. Then $\m$ is new because it is neither $p$-old nor $q$-old. 

If $\ell\geq 5$, the same method as above works and the result follows directly.
\end{proof}

\begin{rem}
In the above proof, we don't need to assume that $\ell$ does not divide $pq$. Note that one direction in the proof of \cite[Theorem 2.4]{Yoo14a} relies on the saturation property of $\T(pq)$ in $\End(J_0(pq))$ locally at $\m$. If either $\ell=p$ or $\ell=q$, this property follows from the second case of \cite[Theorem 3.3]{Yoo14a} because $T_{\ell} \equiv 1 \pmod {\m}$ (cf. \cite[Lemma 1.1]{R08} or \cite[Remark 3.5]{Yoo14a}). The other direction follows by the same argument as in the proof of \cite[Theorem 2.4]{Yoo14a} without further difficulties.
\end{rem}

\subsection{Maximality of $\m_2$}
\begin{thm}\label{thm:m2}
The ideal $\m_2$ cannot be maximal.
\end{thm}
\begin{proof}
For an Eisenstein maximal ideal $\m$, we have $T_{\ell}\equiv 1 \pmod {\m}$. Therefore $\m_2$ is not maximal if either $\ell=p$ or $\ell=q$ because $\ell$ is odd. 
Thus, we assume that $\ell$ does not divide $pq$ and $\m_2$ is maximal. By \cite[Theorem 1.2.(3)]{Yoo14a}, we have $p\equiv q \equiv -1 \modl$ and $\m_2=(\ell, ~U_p-p,~U_q-q,~I_0)$. By \cite[Proposition 5.5]{Yoo15a}, $\m_2$ cannot be maximal, which is a contradiction. Therefore the result follows.
\end{proof}

\subsection{Maximality of $\m_3$ and $\m_4$}\label{sec:m3}
\begin{thm}\label{thm:m3}
The ideal $\m_3$ is maximal if and only if $\ell$ divides the numerator of $\frac{q+1}{(3, ~p(p+1))}$.
By symmetry, the ideal $\m_4$ is maximal if and only if $\ell$ divides the numerator of $\frac{p+1}{(3, ~q(q+1))}$.
\end{thm}
\begin{proof}
Let $\m:=\m_3$ and $I:=I_3$.

If $\ell\geq 5$ and $\ell$ does not divide $pq$, this is proved by Ribet \cite[Theorem 1.4(2)]{Yoo14a}. (The proof of this theorem is given in \textsection 4 of \textit{op. cit.})

Let $n$ be the numerator of $\frac{q+1}{\gcd(3, ~p(p+1))}$. Since $\T$ is a quotient of $\T(pq)$ and the index of $I$ in $\T(pq)$ is equal to $n$ up to powers of 2 \cite[Theorem 3.4]{Yoo14}, $\m$ is not maximal if $\ell$ does not divide $n$. Conversely, if $\ell$ divides $n$ then $S_q[\m] \neq 0$ by Proposition \ref{prop:skoro} below, where $S_q$ is the Skorobogatov subgroup of $J^{pq}$ from the level structure at $q$. Therefore $\m$ is maximal.
\end{proof}

\section{The structure of $J^{pq}[\m]$}\label{sec:structure}
Let $J:=J^{pq}$. In this section, we discuss the structure of $J[\m]$, where $\m=\m_i$ for $1\leq i \leq 4$. 
If $\m_i$ is not maximal, then $J[\m]=0$. Therefore it suffices to study $J[\m_1]$ and $J[\m_3]$ by symmetry.

\subsection{Multiplicity one for Jacobians of Shimura curves}
In this subsection, we prove multiplicity one result as follows. 
\begin{thm}\label{thm:multiplicityone}
Assume that $\m=\m_3$ is maximal. 
If $\ell=3$, we further assume that $3$ does not divide $(p-1)(q-1)$.
Then, $J[\m]$ is a non-trivial extension of $\zell$ by $\muell$. Moreover, $J[\m]$ is ramified at $p$ but is unramified at $q$. Therefore we have $\zell \nsubseteq J[\m]$.
\end{thm}

When $\m=\m_1$ is maximal, the structure of $J[\m]$ is more complicated than one of $J[\m_3]$. However, if one of $p-1$ and $q-1$ is not divisible by $\ell$, then we have the similar result as above. Note that the theorem below is not used in the proof of our main theorem.
\begin{thm}\label{thm:multionem=1}
Assume that $\m=\m_1$ is maximal. Assume further $\ell\geq 5$ and $q \not\equiv 1 \modl$. Then, $J[\m]$ is of dimension $2$ and is ramified at $p$.
\end{thm}

\begin{proof}[Proof of Theorem \ref{thm:multiplicityone}]
Let $\m=\m_3$. By Theorem \ref{thm:m3}, $\m$ is maximal if and only if $\ell$ divides the numerator of $\frac{q+1}{\gcd(3, ~p(p+1))}$. Hence in particular, we assume that $q\not\equiv 1 \modl$. 

For the Jacobian $J_0^D(N)$ with $N$ square-free, we denote by $J_0^D(N)_{/{\F_p}}$ be the special fiber of the N\'eron model of $J_0^D(N)$ over $\F_p$. If $p$ is a divisor of $N$ (resp. of $D$), then it is given by a Deligne-Rapoport model \cite{Bu97, DR73} (resp. a Cerednik-Drinfeld model \cite{Ce76, Dr76}) and the theory of Raynaud \cite{Ra70}. We denote by $\Phi_p(J_0^D(N))$ (resp. $X_p(J_0^D(N))$) the component group of $J_0^D(N)_{/{\F_p}}$ (resp. the character group of $J_0^D(N)_{/{\F_p}}$). 

We shall carry out a proof by several steps. 
\begin{itemize}
\item Step 1 : We show that $\Phi_p(J)[\m]=0$ as follows. 

By Ribet \cite[Theorem 4.3]{R90}, there is a Hecke-equivariant exact sequence:
$$
\xymatrix{
0 \ar[r] & K \ar[r] & (X\oplus X)/{\delta_p (X\oplus X)} \ar[r] & \Phi_p(J) \ar[r] & C \ar[r] & 0,
}
$$
where $X:=X_q(J_0(q))$ and $\delta_p = \mat {p+1}{T_p}{T_p}{p+1}$; and
$K$ (resp. $C$) is the kernel (resp. the cokernel) of the map
$$
\gamma_p : \Phi_q(J_0(q)) \times \Phi_q(J_0(q)) \rightarrow \Phi_q(J_0(pq))
$$
induced by the degeneracy map $\gamma_p : J_0(q) \times J_0(q) \rightarrow J_0(pq)$.
Since $q \not\equiv 1 \modl$, there is no Eisenstein ideal of level $q$.
Therefore the first and second terms of the above exact sequence have no support at $\m$.

\begin{itemize}
\item 
If $\ell\geq 5$, then $C[\m]=0$ by \cite[Proposition A.5]{Yoo14a} and \cite[Corollary A.6]{Yoo14a}.
Therefore $\Phi_p(J)[\m]=0$.
\item
If $\ell=3$ and $q>3$, then the $3$-primary part of $\Phi_q(J_0(pq))$ is cyclic by \cite[\textsection 4.4.1]{Ed91} because $p\not\equiv 1 \pmod {3}$.  Since $U_q$ acts as $1$ on it (cf. \cite[Proposition A.2]{Yoo14a}), we have $C[\m]=0$ and hence $\Phi_p(J)[\m]=0$.
\item
If $\ell=3$ and $q=2$, we have $\Phi_p(J)[\m]=0$ by the table in \cite[p. 210]{Og85} because $p\equiv -1\pmod 3$.  
\end{itemize}

\item Step 2 : We show that $\T_{\m}$ is Gorenstein as follows.

Let $Y=X_p(J)$ and $L=X_q(J_0(pq))$. By Ribet \cite{R90}, there is a Hecke-equivariant exact sequence:
$$
\xymatrix{
0 \ar[r] & Y \ar[r] & L \ar[r] & X \oplus X \ar[r] & 0.
}
$$
By taking completions at $\m$, we have $Y_{\m} \simeq L_{\m}$. By \cite[Theorem 4.5.(4)]{Yoo14}, the dimension of $J_0(pq)[\m]$ is either 2 or 3. However, the dimension of $L/{\m L}$ is 1 in both cases. Therefore, $Y/{\m Y}$ is of dimension 1 as well and $Y_{\m}$ is free of rank 1 over $\T_{\m}$.
Moreover by the monodromy exact sequence, we have a Hecke-equivariant exact sequence:
$$
\xymatrix{
0 \ar[r] & Y \ar[r] & \Hom(Y, \,\Z) \ar[r] & \Phi_p(J) \ar[r] & 0.
}
$$
Since $\Phi_p(J)[\m]=0$, we have $Y_{\m} \simeq \Hom(Y_{\m}, \,\Z_{\ell})$. In other words, $Y_{\m}$ is a free self-dual $\T_{\m}$-module of rank 1, and hence $\T_{\m}$ is Gorenstein.

\item Step 3 : We show that $J[\m]$ is of dimension 2 as follows.

By Grothendieck \cite{Gro72}, there is an exact sequence:
$$
\xymatrix{
0 \ar[r] & \Hom(Y/{\ell^n Y}, \,\mu_{\ell^n}) \ar[r] & J[\ell^n] \ar[r] & Y/{\ell^n Y} \ar[r] & 0.
}
$$
(For details, see \cite[\textsection 3.3]{R76}.)
By taking projective limits, we have
$$
\xymatrix{
0 \ar[r] & \Hom(Y_{\ell}, \,\Z_{\ell}(1)) \ar[r] & \Ta_{\ell} J \ar[r] & Y_{\ell} \ar[r] & 0,
}
$$
where $\Ta_{\ell}J$ is the $\ell$-adic Tate module of $J$ and $\Z_{\ell}(1)$ is the Tate twist of $\Z_{\ell}$. Since $\T_{\m}$ is a direct factor of $\T_{\ell}$, we have 
$$
\xymatrix{
0 \ar[r] & \Hom(Y_{\m}, \,\Z_{\ell}(1)) \ar[r] & \Ta_{\m} J \ar[r] & Y_{\m} \ar[r] & 0.
}
$$
Since $Y_{\m}$ is a free self-dual $\T_{\m}$-module of rank 1, $\Ta_{\m} J$ is a free $\T_{\m}$-module of rank 2. Therefore $J[\m]$ is of dimension 2.

\item Step 4 : We show that $J[\m]$ is ramified at $p$ as follows.

Let $I_p$ be an inertia subgroup of $\GQ$ at $p$. Then by Serre-Tate \cite{ST68}, we have $J[\m]^{I_p} \simeq J_{/{\F_p}}[\m]$. Since $\Phi_p(J)[\m]=0$ and $J^0[\m] = \Hom(Y/{\m Y},\, \muell)$ is of dimension 1, $J[\m]^{I_p} \simeq J^0[\m]$ is of dimension 1 as well, where $J^0$ is the identity component of $J_{/{\F_p}}$. Therefore $J[\m]$ is ramified at $p$.

\item Step 5 :  We show that $J[\m]$ contains $\muell$ as follows.

By Proposition \ref{prop:skoro} below, our assumption on $\m$ implies $\muell \simeq S_q[\m] \subseteq J[\m]$, where $S_q$ is the Skorobogatov subgroup of $J$ from the level structure at $q$.

\item Step 6 : We show that $J[\m]$ is a non-trivial extension of $\zell$ by $\muell$ as follows.

Since all Jordan-H\"older factors of $J[\m]$ are either $\muell$ or $\zell$ 
(cf. \cite[Proposition 14.1]{M77}), the quotient
$J[\m]/{\muell}$ is isomorphic to either $\zell$ or $\muell$. If $J[\m]/{\muell} \simeq \muell$, then $J[\m^{\infty}]$ is a multiplicative $\m$-divisible module, which is a contradiction. Therefore $J[\m]$ is an extension of $\zell$ by $\muell$. Since $J[\m]$ is ramified at $p$, we have $\zell \nsubseteq J[\m]$.

\item Step 7 : We finish the proof by showing that $J[\m]$ is unramified at $q$ as follows.

Let $\Frob_q$ be the Frobenius endomorphism in characteristic $q$. Then, $\Frob_q$ acts by $qU_q$ on the torus $T$ of $J_{/{\F_q}}$ (cf. \cite[Theorem 3.1]{JL86}, \cite{R89b}). Since $U_q \equiv -1 \pmod {\m}$, $T[\m]$ cannot contain $\muell \subseteq J[\m]^{I_q}$.
Since $\T$ acts faithfully on $X_q(J)$ and $\m$ is maximal, $T[\m]$ is at least of dimension 1. Therefore $J[\m]^{I_q} \simeq J_{/{\F_q}}[\m]$ is at least of dimension 2. Thus, $J[\m]$ is unramified at $q$.
\end{itemize}
\end{proof}

\begin{rem}
Most of the above proof was given by Ribet in \cite[Appendix B]{Yoo14a} with the assumption that $p \not\equiv 1\modl$ and $\ell\geq 5$. However we duplicate the proof here to point out where our assumption plays a role.
\end{rem}

\begin{proof}[Proof of Theorem \ref{thm:multionem=1}] Since $q\not\equiv 1\modl$, $\m$ is not $p$-old by Mazur \cite{M77}. Moreover we have $C[\m]=0$ as in Step 1 of the above proof because $U_q$ acts by $q$ on $C[\ell]$. Therefore the argument in Step 1 works in this case as well. With our assumption on $q$, the arguments in Steps 2--4 are also valid as above, and hence the result follows.
\end{proof}

\subsection{The Skorobogatov subgroups of $J$}\label{sec:Skorobogatov}
In this subsection, we discuss a subgroup of $J[\m_3]$, which is the $\ell$-torsion subgroup of the Skorobogatov subgroup $S_q$ from the level structure at $q$. 
In \cite[Appendix C]{Yoo14a}, we studied the actions of the Hecke operators on $S_q$ and computed its order \upto. Since we include the discussion with $\ell=3$, we compute the $\ell$-torsion subgroup on $S_q$ for any odd prime $\ell$. 
\begin{prop}\label{prop:skoro}
We have $S_q[\ell] \neq 0$ if and only if $\ell$ divides the numerator of $\frac{q+1}{\gcd(3, ~p(p+1))}$. If $S_q[\ell]\neq 0$, then we have $S_q[\ell]=S_q[\m_3]\simeq \muell$.
\end{prop}
\begin{proof}
Since the order of $S_q$ is equal to $\frac{q+1}{\epsilon(q)}$ (up to powers of $2$), the first statement follows by the definition of $\epsilon(q)$ in \cite[p.~781]{Sk05}. 
Since $S_q$ is the Cartier dual of the constant cyclic group scheme (cf. \textit{loc. cit.}), $S_q[\ell]$ is isomorphic to $\muell$ if it is not zero. Therefore we have $S_q[\ell]=S_q[\m_3]\simeq \muell$ by \cite[Proposition C.2]{Yoo14a} if $S_q[\ell]\neq 0$. 
\end{proof}

\section{Non-existence of rational points of order $\ell$ on $J^{pq}$}\label{sec:proof}
In this section, we prove our main theorem.
\begin{thm}
For a prime $\ell\geq 5$, the Jacobian $J^{pq}$ does not have rational points of order $\ell$ unless one of the following holds:
\begin{itemize}
\item $p\equiv q \equiv 1 \modl$; 
\item $p\equiv 1 \modl$ and $q^{\frac{p-1}{\ell}} \equiv 1 \pmod p$;
\item $q\equiv 1 \modl$ and $p^{\frac{q-1}{\ell}} \equiv 1 \pmod q$.
\end{itemize}
Furthermore, the Jacobian $J^{pq}$ does not have rational points of order $3$ if $(p-1)(q-1)$ is not divisible by $3$.
\end{thm}
\begin{proof}
Let $A:=J^{pq}(\Q)_{\tor}$ and $A_{\ell}:=A \otimes_{\Z} \Z_{\ell}$. Then $A_{\ell}$ is a $\T_{\ell}$-module. By Eichler-Shimura relation and the isogeny $\Psi(pq)$, which is Hecke-equivariant, for a prime $r$ not dividing $pq$
$$
T_r \equiv \Frob_r+\Ver_r  ~\text{ on }~ J^{pq}_{/{\F_r}},
$$
where $\Frob_r$ is the Frobenius morphism in characteristic $r$ and $\Ver_r$ is its transpose. Therefore $T_r-1-r$ kills $A$ and hence $A_{\ell}$ is annihilated by $I_0$, i.e., $A_{\ell}$ is a $\T_{\ell}/{I_0}$-module. By Proposition \ref{prop:decomposition}, it decomposes into $A_{\ell}^i$, where each $A_{\ell}^i$ is a $\T_{\ell}/{I_i}$-module. More precisely, we have 
$A_{\ell}^i = A_{\ell} \cap J^{pq}[I_i] = A_{\ell}[I_i]$. 
Thus, it suffices to prove that $A_{\ell}^i=0$ for all $1\leq i \leq 4$.

If all the above assumption do not hold, then $\m_1$ is not maximal. Therefore $A_{\ell}^1=0$. By Theorem \ref{thm:m2}, we have $A_{\ell}^2=0$ as well. Now we assume that $A_{\ell}^3 \neq 0$. If $\ell=3$, then we further assume that $\ell$ does not divide $(p-1)(q-1)$. 
Then $A_{\ell}^3[\ell] \simeq (\zell)^a$ for some $a \geq 1$. Since $A_{\ell}^3[\ell]=A_{\ell}[\ell, ~I_3]=A_{\ell}[\m_3]$, we have $\zell \subseteq J^{pq}[\m_3]$. This contradicts Theorem \ref{thm:multiplicityone}. Thus, we have $A_{\ell}^3=0$ and hence $A_{\ell}^4=0$ by symmetry. 
\end{proof}

\section{The kernel of an isogeny due to Ribet}\label{sec:kernelofisogeny}
In this section, we provide an application of our main theorem. As before, let $J_0(pq)^{\new}$ denote the new quotient of $J_0(pq)$,
$\Psi(pq)$ denote an isogeny from $J_0(pq)^{\new}$ to $J^{pq}$, and let $K(pq)$ denote the kernel of $\Psi(pq)$:
$$
\xymatrix{
0 \ar[r] & J_0(pq)_{\old} \ar[r] & J_0(pq) \ar[r]^-{\pi} & J_0(pq)^{\new} \ar[r] & 0;
}
$$
\vspace{-7mm}
$$
\xymatrix{
0\ar[r] & K(pq) \ar[r] & J_0(pq)^{\new} \ar[r]^-{\Psi(pq)} & J^{pq} \ar[r] & 0.
}
$$
Ogg \cite{Og85} conjectured that the image of some cuspidal divisors in $J_0(pq)$ is contained in $K(pq)$. This conjecture is proved by Gonz\'alez and Molina \cite{GM11} if the genus of $\cX^{pq}$ is at most $3$ . We prove some of the conjecture by Ogg as follows:

\begin{thm} 
Let $\ell^m$ and $\ell^n$ be the exact powers of $\ell$ dividing $p+1$ and $q+1$, respectively.
If $\ell\geq 5$ and all the following conditions hold, then $K(pq)$ contains $\pi(\cC_{\ell}(pq))$, and the latter is isomorphic to $\zmod {\ell^m} \oplus \zmod {\ell^n}$:
\begin{itemize}
\item $\ell$ does not divide $(p-1, q-1)$; 
\item if $p\equiv 1 \modl$, then $q^{\frac{p-1}{\ell}} \not\equiv 1 \pmod p$;
\item if $q\equiv 1 \modl$, then $p^{\frac{q-1}{\ell}} \not\equiv 1 \pmod q$.
\end{itemize}

If $\ell=3$ and $(p-1)(q-1)$ is not divisible by $3$, then $K(pq)$ contains $\pi(\cC_3(pq))$, and the latter is isomorphic to $\zmod {3^{\alpha}} \oplus \zmod {3^{\beta}}$,  where $\alpha=\max \{0,~m-1\}$ and $\beta=\max \{0, ~n-1\}$.
\end{thm}

\begin{proof} 
Let $C_p:=[P_1-P_p]$ and $C_q:=[P_1-P_q]$ be elements in $\cC(pq)$, where $P_t$ is the cusp of $X_0(pq)$ corresponding to $1/t \in \bP^1(\Q)$. 

Assume that $\ell\geq 5$. 
Let $(p-1)(q^2-1)=\ell^a \times x$ and $(q-1)(p^2-1)=\ell^b \times y$, where $\ell$ does not divide $xy$. Let $D_p:=xC_p$ and $D_q:=yC_q$. Assume that all the above three conditions hold. Then by Chua-Ling \cite{CL97}, we have $\cC_{\ell}(pq) \simeq \br{D_p} \oplus \br{D_q}$ and it is contained in $J_0(pq)(\Q)_{\tor}$. By symmetry, we may assume that $q\not\equiv 1 \modl$. Then, the intersection of $\cC_{\ell}(pq)$ and $J_0(pq)_{\old}$ is isomorphic to $\br {\ell^n D_p} \oplus \br {\ell^m D_q}$ (cf. \cite[Theorem 2]{CL97}). Thus, $\pi(C_{\ell}(pq)) \simeq \zmod {\ell^n} \oplus \zmod {\ell^m}$. Since $J^{pq}(\Q)_{\tor, ~\ell}=0$ by Theorem \ref{thm:maintheorem}, $K(pq)$ contains $\pi(C_{\ell}(pq))$.

Assume that $\ell=3$ and $3$ does not divide $(p-1)(q-1)$.
Note that the order of $C_p$ (resp $C_q$) is the numerator of $\frac{(p-1)(q^2-1)}{3}$ (resp. $\frac{(q-1)(p^2-1)}{3}$) up to powers of $2$. Thus, $\cC_3(pq)$ is isomorphic to $\zmod {3^{\alpha}}\oplus \zmod {3^{\beta}}$. Since the $3$-primary subgroups of the rational torsion subgroups of $J_0(pq)_{\old}$ and $J^{pq}$ are zero, $K(pq)$ contains $\pi(C_3(pq))$, and the latter is isomorphic to $\zmod {3^{\alpha}}\oplus \zmod {3^{\beta}}$.
\end{proof}

\begin{rem}
Let $p$ and $q$ be distinct primes with $p<q$ 
and let $S$ be the set of pairs $(p,~q)$ such that $g(\cX^{pq}) \leq 3$. In this case, Gonz\'alez and Molina determined the kernel of $K(pq)$ by taking some precise isogeny between $J_0(pq)^{\new}$ and $J^{pq}$.
Let $S_{\ell}$ be the subset of $S$ consisting of the pairs satisfying all the above three conditions with respect to $\ell$. Then, the following table describes the orders of $K(pq)$ (for their chosen $\Psi(pq)$) and 
$$
\cD(pq) := \bigoplus_{\substack{\ell~\text{odd primes}\\\text{such that }(p, ~q) \in S_{\ell}}} 
\pi(\cC_{\ell}(pq)).
$$
If $\ell$ is large enough, then $\cC_{\ell}(pq)=0$ and hence the direct sum in the definition is actually a finite sum. Moreover, from its definition and the above theorem, $\cD(pq) \subseteq \pi(\cC(pq)) \cap K(pq)$. We can see that $K(pq)/\cD(pq)$ is a $2$-group for any $(p,~q) \in S$ from the table below.
\end{rem}

\begin{center}
\begin{tabular}{| c | c | c | c | c | c | c |}
\hline
$S$ & $g(\cX^{pq})$ & $\in S_3$? &$\in S_5$?&$\in S_7$?&$\# \cD(pq)$& $\# K(pq)$ \\ \hline
$(2,~7)$ &1 &No &Yes &Yes & 1&2 \\ \hline
$(2,~17)$ &1 &Yes &Yes &Yes &3& 3\\ \hline
$(3,~5)$ &1 &Yes &Yes &Yes &1 &1 \\ \hline
$(3,~7)$ &1 &No &Yes &Yes & 1&2 \\ \hline
$(3,~11)$ &1 &Yes &Yes &Yes &1 &1 \\ \hline
$(2,~13)$ &2 &No &Yes &Yes & 7&7 \\ \hline
$(2,~19)$ &2 &No &Yes &Yes & 5&5 \\ \hline
$(2,~29)$ &2 &Yes &Yes &Yes & 5& 5\\ \hline
$(2,~31)$ &3 &No &Yes &Yes &1 &8 \\ \hline
$(2,~41)$ &3 &Yes &Yes &Yes & 7&7 \\ \hline
$(2,~47)$ &3 &Yes &Yes &Yes & 1&4 \\ \hline
$(3,~13)$ &3 &No &Yes &Yes & 7&7 \\ \hline
$(3,~17)$ &3 &Yes &Yes &Yes & 3&3 \\ \hline
$(3,~19)$ &3 &No &Yes &Yes & 5&20 \\ \hline
$(3,~23)$ &3 &Yes &Yes &Yes & 1&8 \\ \hline
$(5,~7)$ &3 &No &Yes &Yes & 1&2 \\ \hline
$(5,~11)$ &3 &Yes &Yes &Yes & 1&1 \\ \hline
\end{tabular}\\
\vspace{2mm}
Table 1.
\end{center}

\end{document}